\documentclass[11pt,a4paper]{article}
\usepackage{amsmath}
\usepackage{amsthm}
\usepackage{amssymb}
\usepackage{amscd}
\usepackage{graphicx}
\usepackage{epsfig}
\usepackage[matrix,arrow,curve]{xy}

\usepackage{latexsym}
\usepackage{amsfonts}
\input xy
\usepackage{tikz}
\usetikzlibrary{arrows,chains,matrix,positioning,scopes,snakes}
\makeatletter
\tikzset{join/.code=\tikzset{after node path={%
\ifx\tikzchainprevious\pgfutil@empty\else(\tikzchainprevious)%
edge[every join]#1(\tikzchaincurrent)\fi}}}
\makeatother

\tikzset{>=stealth',every on chain/.append style={join},
         every join/.style={->}}
\tikzset{
    >=stealth',
    punkt/.style={
           rectangle,
           rounded corners,
           draw=black, very thick,
           text width=6.5em,
           minimum height=2em,
           text centered},
    pil/.style={
           ->,
           thick,
           shorten <=2pt,
           shorten >=2pt,}
}

\xyoption{all} \tolerance=500

\newtheorem{thm}{Theorem}[section]
\newtheorem{prop}[thm]{Proposition}
\newtheorem{lem}[thm]{Lemma}
\newtheorem{cor}[thm]{Corollary}

\theoremstyle{definition}
\newtheorem{example}[thm]{Example}
\newtheorem{remark}[thm]{Remark}
\newtheorem{definition}[thm]{Definition}

\font\black=cmbx10 \font\sblack=cmbx7 \font\ssblack=cmbx5 \font\blackital=cmmib10  \skewchar\blackital='177
\font\sblackital=cmmib7 \skewchar\sblackital='177 \font\ssblackital=cmmib5 \skewchar\ssblackital='177
\font\sanss=cmss12 \font\ssanss=cmss8 scaled 900 \font\sssanss=cmss8 scaled 600 \font\blackboard=msbm10
\font\sblackboard=msbm7 \font\ssblackboard=msbm5 \font\caligr=eusm10 \font\scaligr=eusm7 \font\sscaligr=eusm5

\font\bsymb=cmsy10 scaled\magstep2
\def\all#1{\setbox0=\hbox{\lower1.5pt\hbox{\bsymb
       \char"38}}\setbox1=\hbox{$_{#1}$} \box0\lower2pt\box1\;}
\def\exi#1{\setbox0=\hbox{\lower1.5pt\hbox{\bsymb \char"39}}
       \setbox1=\hbox{$_{#1}$} \box0\lower2pt\box1\;}

\newfam\bifam
\textfont\bifam=\blackital \scriptfont\bifam=\sblackital \scriptscriptfont\bifam=\ssblackital

\newfam\blfam
\textfont\blfam=\black \scriptfont\blfam=\sblack \scriptscriptfont\blfam=\ssblack

\newfam\bbfam
\textfont\bbfam=\blackboard \scriptfont\bbfam=\sblackboard \scriptscriptfont\bbfam=\ssblackboard

\newfam\ssfam
\textfont\ssfam=\sanss \scriptfont\ssfam=\ssanss \scriptscriptfont\ssfam=\sssanss
\def\sss#1{{\fam\ssfam\relax#1}}

\newfam\clfam
\textfont\clfam=\caligr \scriptfont\clfam=\scaligr \scriptscriptfont\clfam=\sscaligr

\def\pmb#1{\setbox0\hbox{${#1}$} \copy0 \kern-\wd0 \kern.2pt \box0}
\def\pmbb#1{\setbox0\hbox{${#1}$} \copy0 \kern-\wd0
      \kern.2pt \copy0 \kern-\wd0 \kern.2pt \box0}
\def\pmbbb#1{\setbox0\hbox{${#1}$} \copy0 \kern-\wd0
      \kern.2pt \copy0 \kern-\wd0 \kern.2pt
    \copy0 \kern-\wd0 \kern.2pt \box0}
\def\pmxb#1{\setbox0\hbox{${#1}$} \copy0 \kern-\wd0
      \kern.2pt \copy0 \kern-\wd0 \kern.2pt
      \copy0 \kern-\wd0 \kern.2pt \copy0 \kern-\wd0 \kern.2pt \box0}
\def\pmxbb#1{\setbox0\hbox{${#1}$} \copy0 \kern-\wd0 \kern.2pt
      \copy0 \kern-\wd0 \kern.2pt
      \copy0 \kern-\wd0 \kern.2pt \copy0 \kern-\wd0 \kern.2pt
      \copy0 \kern-\wd0 \kern.2pt \box0}

\mathchardef\za="710B  
\mathchardef\zb="710C  
\mathchardef\zg="710D  
\mathchardef\zd="710E  
\mathchardef\zve="710F 
\mathchardef\zz="7110  
\mathchardef\zh="7111  
\mathchardef\zvy="7112 
\mathchardef\zi="7113  
\mathchardef\zk="7114  
\mathchardef\zl="7115  
\mathchardef\zm="7116  
\mathchardef\zn="7117  
\mathchardef\zx="7118  
\mathchardef\zp="7119  
\mathchardef\zr="711A  
\mathchardef\zs="711B  
\mathchardef\zt="711C  
\mathchardef\zu="711D  
\mathchardef\zvf="711E 
\mathchardef\zq="711F  
\mathchardef\zc="7120  
\mathchardef\zw="7121  
\mathchardef\ze="7122  
\mathchardef\zy="7123  
\mathchardef\zf="7124  
\mathchardef\zvr="7125 
\mathchardef\zvs="7126 
\mathchardef\zf="7127  
\mathchardef\zG="7000  
\mathchardef\zD="7001  
\mathchardef\zY="7002  
\mathchardef\zL="7003  
\mathchardef\zX="7004  
\mathchardef\zP="7005  
\mathchardef\zS="7006  
\mathchardef\zU="7007  
\mathchardef\zF="7008  
\mathchardef\zW="700A  

\newcommand{\be}{\begin{equation}}
\newcommand{\ee}{\end{equation}}
\newcommand{\raa}{\rightarrow}

\newcommand{\bea}{\begin{eqnarray}}
\newcommand{\eea}{\end{eqnarray}}
\newcommand{\beas}{\begin{eqnarray*}}
\newcommand{\eeas}{\end{eqnarray*}}
\def\*{{\textstyle *}}
\newcommand{\T}{{\mathbb T}}

\newcommand{\w}{\wedge}
\newcommand{\nn}{\nonumber}

\newcommand{\ti}{\times}

\def\ran{\rangle}

\def\cA{{\cal A}}

\def\cC{{\cal C}}

\def\cN{{\cal N}}

\def\cJ{{\cal J}}

\def\cO{{\cal O}}
\def\cU{{\cal U}}

\def\sT{{\sss T}}

\def\cM{{\cal M}}

\newdir{|>}{%
!/4.5pt/@{|}*:(1,-.2)@^{>}*:(1,+.2)@_{>}}

\newcommand{\mZ}{\mathbb{Z}_2}

\newcommand{\p}{\partial}
\newcommand{\la}{\langle}

\newcommand{\Ci}{C^{\infty}}

\newcommand{\N}{\mathbb{N}}
\newcommand{\Z}{\mathbb{Z}}
\newcommand{\R}{\mathbb{R}}

\newcommand{\ao}{\mathbf{1}^n}

\newcommand{\lp}{\left(}
\newcommand{\rp}{\right)}

\newcommand{\op}[1]{\!\!\mathop{\rm ~#1}\nolimits}

\newcommand\dirac{/\kern-1.2ex\partial}
\newcommand\Dc{/\kern-1.5ex D}

\begin{document}
\title{\bf The category of $\mathbb{Z}_2^n$-supermanifolds}
\date{}
\author{Tiffany Covolo\footnote{National Research University Higher School of Economics, Moscow, tcovolo@hse.ru}, Janusz Grabowski\footnote{Polish Academy of Sciences, Warsaw, jagrab@impan.pl}, and Norbert Poncin\footnote{University of Luxembourg, norbert.poncin@uni.lu}}

\maketitle

\begin{abstract}{In Physics and in Mathematics $\Z_2^n$-gradings, $n\ge 2$, appear in various fields. The corresponding sign rule is determined by the `scalar product' of the involved $\Z_2^n$-degrees.
The $\Z_2^n$-Supergeometry exhibits challenging differences with the classical one: nonzero degree even coordinates are not nilpotent, and even (resp., odd) coordinates do not necessarily commute (resp., anticommute) pairwise.
In this article we develop the foundations of the theory: we define $\Z_2^n$-supermanifolds and provide examples in the ringed space and coordinate settings. We thus show that formal series are the appropriate substitute for nilpotency. Moreover, the class of $\Z_2^\bullet$-supermanifolds is closed with respect to the tangent and cotangent functors. We explain that any $n$-fold vector bundle has a canonical `superization' to a $\Z_2^n$-supermanifold and prove that the fundamental theorem describing supermorphisms in terms of coordinates can be extended to the $\Z_2^n$-context.}
\end{abstract}

\vspace{2mm} \noindent {\bf MSC 2010}: 17A70, 58A50, 13F25, 16L30 \medskip


\thispagestyle{empty}

\section{Introduction}

{\bf Motivation and discussion}. Classical Supersymmetry and Supergeometry are not sufficient to suit the current needs. In {Physics}, $\Z_2^n$-gradings, $n\ge 2,$ are used in string theory and in parastatistics \cite{AFT10}, \cite{YJ01}. In {Mathematics}, there exist good examples of $\Z_2^n$-graded \emph{$\Z_2^n$-commutative algebras} (i.e., the superscript of $-1$ in the sign rule is the standard `scalar product' of $\Z_2^n$): quaternions and, more generally, any Clifford algebra, the algebra of Deligne differential superforms, etc. Moreover, there exist interesting and canonical examples of \emph{$\Z^n_2$-supermanifolds}.

For instance, note that the tangent bundle of a classical $\Z_2$-supermanifold $\cM$ can be viewed as a $\Z_2$-supermanifold in two different ways, as $\sT M$ or $\sT[1]\cM$, and as a $\Z_2^2$-supermanifold $\T\cM$, in which case the function sheaf is the completion of differential superforms of $\cM$ together 
with the Deligne sign convention. Actually the tangent (and cotangent) bundle of any $\Z_2^n$-supermanifold is a $\Z_2^{n+1}$-supermanifold. Further, any double vector bundle (resp., any \emph{$n$-fold vector bundle}) canonically provides a $\Z_2^2$-supermanifold (resp., $\Z_2^n$-supermanifold) as its `superization'.

To be more precise, suppose that $\cM$ is a supermanifold with local coordinates $(x^1,\dots,x^p,$ $\zx^1,\dots,\zx^q)$, where the $x^i$ are even and the $\zx^a$ odd. For the tangent bundle $\sT\cM$, with the adapted local coordinates $(x^i,\zx^a,\dot x^j,\dot\zx^b)$, one can introduce a supermanifold structure declaring $\dot x^j$ to be even and $\dot\zx^b$ to be odd, or reversing these parities for $\sT[1]\cM$.
For the latter, the variables $\dot\zx^b$ are even; they are true real-valued variables although the $\zx^b$ are indeterminates or formal variables.
Hence, the local model of the function algebra of the corresponding supermanifold $\sT[1]\cM$ is given by the polynomials in the odd indeterminates $\zx^a$ and $\dot x^j$ with coefficients in the smooth functions with respect to the even variables $x^i$ and $\dot \zx^b$. These polynomials are referred to as pseudodifferential forms. Pseudodifferential forms can be defined not only locally on superdomains but also globally on supermanifolds \cite{Lei11}. They have been introduced in \cite{BL77} to obtain objects suitable for integration. The algebra $\widehat{\zW}(\cM)$ of pseudodifferential forms on $\cM$ is exactly the function algebra of the supermanifold $\sT[1]\cM$.

On the other hand, the tangent bundle, as every vector bundle, admits an $\N$-grading for which $\dot x^j$ and $\dot\zx^b$ are of degree 1 (and $x^i$ and $\zx^a$ are of degree 0). Thus we have a canonical bigrading by the monoid $\N\times \mZ$, which can be reduced to $\mZ^2=\mZ\ti\mZ$. With respect to this bigrading, $(x^i,\zx^a,\dot x^j,\dot\zx^b)$ are of bidegrees $(0,0),(0,1),(1,0)$, and $(1,1)$, respectively. In particular, the variables $\dot\zx^b$ are even but have nonzero degree, so that, in accordance with a common feeling, we cannot claim that they are real variables \cite[Section 3.3.2]{Wi12}.

Now, any symmetric bi-additive map $\la -,- \ran:\mZ^2\ti\mZ^2\to\Z$ gives rise to a sign rule: $$AB=(-1)^{\la(m,n),(k,l)\ran}BA\;,$$ where $A$ and $B$ are coordinates of bidegrees $(m,n)$ and $(k,l)$, respectively. We get the sign rule for nonreversed parity when choosing $\la(m,n),(k,l)\ran=n\,l$, obtain the \emph{Bernstein-Leites sign rule} (the sign rule for reversed parity used in $\sT[1]\cM$) when setting
$\la(m,n),(k,l)\ran=(m+n)(k+l)$, and we find the \emph{Deligne sign rule} for the `scalar product' $\la(m,n),(k,l)\ran=mk+nl$; see discussion in \cite[Appendix to \S 1]{DM99}. Note that the latter does not lead to a {\em superalgebra}, as the $\dot\zx^b$ are even in the sense that they commute among themselves, but anticommute with the $\zx^a$. To include Deligne's convention into the picture which, as just mentioned, does not correspond to any \emph{supermanifold}, it is natural to extend the notion of supermanifold and to admit $\Z_2^2$-gradings, or even $\Z_2^n$-gradings with $n\ge 2$, as well as the corresponding sign rule
\be\label{Z_2^n-com}AB=(-1)^{\sum_{i=1}^n m_ik_i}BA\;,
\ee
where $(m_1,\ldots,m_n)$ and $(k_1,\ldots,k_n)$ are the $\Z_2^n$-degrees of $A$ and $B$. The above additional canonical grading, together with Deligne's convention, makes the tangent (as well as the cotangent) bundle into a $\Z_2^2$-supermanifold $\T\cM$ (resp. $\T^\star\cM$), so that the class of $\Z_2^\bullet$-supermanifolds is closed with respect to the tangent (resp., cotangent) functor.

Thus we get the hierarchy: $\,\Z_2^0$-Supergeometry (classical even Differential Geometry) contains the germ of $\Z_2^1$-Supergeometry (standard Supergeometry), which in turn contains the sprout of $\Z_2^2$-Supergeometry, etc. Moreover, the $\Z_2^n$-supergeometric viewpoint provides deeper insight and simplified solutions, just as the classical supergeometric approach does in comparison with ordinary differential Geometry. Interesting relations with Quantum Mechanics, Quantum Field Theory and String Theory are expected.
\medskip

Although not universally accepted at the beginning, $\Z_2^n$-Supergeometry is, in view of what has been said, a {\bf necessary and natural} generalization. When defining the parity of a $\Z_2^n$-degree as the parity of the total degree, nonzero degree even coordinates are not nilpotent, and even (resp., odd) coordinates do not necessarily commute (resp., anticommute) pairwise. These circumstances lead to interesting differences with the classical theory.

The reason for initial skepticism was Neklyudova's equivalence \cite{Lei11}: this result states that the categories of $\Z_2^n$-graded {\it $\Z_2^n$-commutative} and $\Z^n_2$-graded {\it supercommutative} algebras are equivalent. However, working with $\Z_2^n$-supermanifolds means that we are working with very particular $\Z_2^n$-superalgebras, describing geometrical objects really different from standard supermanifolds, and various geometrical constructions, like passing from $\cM$ to $\T\cM$, require that we should be able to consider $\Z_2^n$-supermanifolds for different $n$ at the same footing. Moreover,
it turned out that the construction of `{\em $\Z_2^n$-commutative} concepts' (e.g., the $\Z_2^n$-Berezinian \cite{COP12}) via simple pullbacks of the corresponding `{\it supercommutative} concepts' (the classical Berezinian) is not always as easy as one might expect. Therefore, Neklyudova's theorem {\bf does not ban} investigation of $\Z_2^n$-supermanifolds. On the other hand, we will show that $\Z_2^n$-grading is {\bf sufficient} in the sense that any sign rule, for any finite number $N$ of coordinates, is of the form (\ref{Z_2^n-com}), for some $n\le 2N$.\medskip

\noindent{\bf Possible applications}. The theory of $\Z_2^n$-supermanifolds is closely related to Clifford calculus, see above. Clifford algebras have numerous applications in Physics, but the use of $\Z_2^n$-gradings has never been studied. For instance, Clifford algebras and modules are crucial in understanding Spin-structures on manifolds together with their physical consequences, e.g., the Dirac operator. While the $\Z_2^n$-refined viewpoint led to new results on Clifford algebras and modules over them [COP12], its impact on standard applications in Geometry and Field Theory has still to be explained.

Further, the theory of $\Z_2^n$-manifolds should lead to a novel approach to quaternionic functions: examples of application areas include thermodynamics, hydrodynamics, geophysics and structural mechanics.

It is eventually interesting to observe the parallelism of our $\Z_2^n$-extension with Baez' suggestion of a common generalization -- under the name of $r$-Geometry -- of superalgebras and Clifford algebras with the goal to incorporate, besides bosons and fermions, also anyons into the picture \cite{Bae92}.\medskip

\noindent{\bf Main difference with standard Supergeometry}. The key-concept of $\Z_2^n$-Superalgebra is the (already mentioned) $\Z_2^n$-Berezinian. This higher Berezinian (which is tightly connected with quasi-determinants) and the corresponding (via the Liouville formula) higher trace have been constructed \cite{COP12}. It provides a new solution (`different' from the Dieudonn\'e determinant) to Cayley's challenge to build a determinant of quaternionic matrices.

Hence, the new $\Z_2^n$-Supergeometry not only includes differential but possibly also integral calculus, in contrast with \cite{Mol10} which develops a functorial concept of $\Z_2^n$-supermanifold and mentions explicitly a lack of insight as concerns integration. However, whereas parts of the differential calculus on $\Z_2^n$-supermanifolds are `just' extensions (sometimes nice ones) of the classical theory, integral calculus on $\Z_2^n$-supermanifolds must be completely different from the classical case \cite{Pon16}.\medskip

\noindent{\bf Sources}. Of course, there is an extensive literature on Supergeometry and related topics and it is impossible to give complete references; let us at least mention \cite{BBHR91}, \cite{Ber79}, \cite{Ber87}, \cite{DeW84}, \cite{Kos75}, \cite{Rog07}, \cite{Tuy04}. The sources that had the highest impact on the present text are: \cite{DAL}, \cite{Rog07}, \cite{Tuy04}, \cite{Lei11}, \cite{Var}, \cite{Man}, \cite{DM99}, \cite{CCF}, \cite{DSB}, \cite{Vor12}, \cite{BP12}, \cite{GKP1}, \cite{GKP2}, \cite{Schm97}.
\section{Sign rules}

In standard Supergeometry the sign rule is completely determined by the parity. Why not accept an arbitrary commutation rule? In principle, one can consider a general grading by a semigroup and an arbitrary commutation factor, i.e., work with so-called \emph{colored algebras}.\medskip

More precisely, let $K$ be a commutative unital ring, $K^\times$ be the group of invertible elements
of $K$, and let $G$ be a commutative semigroup. A map $\zf: G \ti G \to K^\times$ is called a \emph{commutation factor} on $G$ if
\be\label{0}\zf(g,h)\zf(h,g) = 1\,,\quad\zf(f,g+h)=\zf(f,g)\zf(f,h)\,,\quad \text{and}\quad \zf(g,g) = \pm 1\,,\ee
for all $f, g, h\in G$.
Note that these axioms imply that
$$\zf(f+g,h)=\zf(f,h)\zf(g,h)$$
and that the condition $\zf(g,g) = \pm 1$ follows automatically from the other two axioms if $K$ is a field.\medskip

Let $\cA$ be a $G$-graded $K$-algebra $\cA=\bigoplus_{g\in G}\cA^g$. Elements $x$ from $\cA^g$ are called \emph{$G$-homogeneous
of degree} or \emph{weight} $g=:\op{deg}(x)$. The algebra $\cA$ is said to be \emph{$\zf$-commutative} if
\be a b = \zf(\op{deg}(a),\op{deg}(b))b a\;,\ee
for all $G$-homogeneous elements $a, b \in \cA$. Homogeneous elements $x$ with $p(\op{deg}(x))=p(g):=\zf(g,g)=-1$
are \emph{odd}, the other homogeneous elements are \emph{even}.
Graded algebras with commutation rules of this kind are known under the name \emph{color algebras} \cite{KS12,Lych95,Sche79}. In this paper we will be interested in color associative and color Lie algebras whose commutation factor is just a sign.\medskip

In what follows, $K$ will be $\R$ and $\zf$ will take the form $$\zf(g,h)=(-1)^{\la g, h\ran}\;,$$ for
a `scalar product' $\la -,- \ran:G\ti G\to\Z$. This means that we use the {\it commutation factor} as the \emph{sign rule}. In this note we confine ourselves to $G=\mZ^n$ and the standard `scalar product' of $\Z_2^n$, what will lead to $\Z_2^n$-Supergeometry with nicer categorical properties than the standard Supergeometry. More precisely, we propose a generalization of differential $\mZ^1$-Supergeometry to the case of a $\mZ^n$-grading in the structure sheaf.\medskip

Indeed, we will show that any sign rule, for any finite number of coordinates, can be obtained from the `scalar product'
\be\label{sp}\la(i_1,\dots,i_n),(j_1,\dots,j_n)\ran_n=i_1j_1+\ldots +i_nj_n
\ee on $\mZ^n$ for a sufficiently big $n$. In other words, any algebra that is finitely generated by some generators satisfying certain sign rules can be viewed as a $\mZ^n$-graded associative algebra $\cA=\bigoplus_{i\in\mZ^n}\cA^i$, which is $\Z_2^n$-commutative in the sense that
$$y^iy^j=(-1)^{\la i,j\ran_n}y^jy^i\;,$$
for all $y^i\in\cA^i$, $y^j\in\cA^j$. We simply refer to such algebras as {$\mZ^n$-\emph{commutative associative algebras}}. Let us mention that a similar theorem was proved for group gradings in \cite{MGO10}.\medskip

Let now $S$ be a finite set, say $S=\{ 1,\dots,m\}$, and let $\zf:S\ti S\to\{\pm 1\}$ be any symmetric function. We can understand $\zf$ as a sign rule for an associative algebra generated by elements $y^i$, $i=1,\dots, m$, i.e.,
$$y^iy^j=\zf(i,j)y^jy^i\,.$$
We then have the following.
\begin{thm}\label{SuffGrad} There is $n\le 2m$ and a map $\zs:S\to\mZ^n$, $i\mapsto\zs_i$, such that \be\label{sr}\zf(i,j)=(-1)^{\la\zs_i,\zs_j\ran_n}\;.
\ee
\end{thm}
\begin{proof}
We interpret $\mZ^{2m}$ as the set of functions $\{\pm 1,\dots,\pm m\}\to\mZ$, and denote by $p(i,j)\in\{0,1\}$ the \emph{parity} of $\zf(i,j)$: $(-1)^{p(i,j)}=\zf(i,j)$.

First, define $\zs_1\in \mZ^{2m}$ by $\zs_1(1)=1$, $\zs_1(-1)=1+p(1,1)\in\mZ$, and $\zs_1(k)=0$ for $|k|>1$. Then,  for $j=2,\dots,m$, define $\zs_j(1)=p(j,1)$ and $\zs_j(-1)=0$. Independently of the definition of the remaining values of $\zs_j$, Condition (\ref{sr}) is valid for $i=1$ and all $j=1,\dots, m$, since $\zs_1(k)=0$ for $|k|>1$.

Assume inductively that we have fixed $\zs_1,\dots,\zs_r$, with $\zs_j(k)=0$ for $|k|>j$, as well as the values $\zs_j(k)$, for $j=r+1,\dots, m$ and $|k|\le r$, so that (\ref{sr}) is valid for $i=1,\dots,r$ and all $j$.

Define:
$$\zs_{r+1}(r+1)=1\,, \quad\zs_{r+1}(-r-1)=1+\sum_{|k|=1}^r\zs_{r+1}(k)+p(r+1,r+1)\,,\quad \zs_{r+1}(k)=0\ \text{for}\ |k|>r+1\,.$$
Then, (\ref{sr}) is valid also for $i=j=r+1$. Putting now $\zs_j(-r-1)=0$ and
$$\zs_j(r+1)=\sum_{|k|=1}^r\zs_j(k)\zs_{r+1}(k)+p(j,r+1)$$
for $j=r+2,\dots, m$, we finish with fixed $\zs_1,\dots,\zs_{r+1}$, with $\zs_j(k)=0$ for $|k|>j$, and the values $\zs_j(k)$, for $j=r+2,\dots, m$ and $|k|\le r+1$, so that (\ref{sr}) is valid for $i=1,\dots,r+1$ and all $j$. This proves the inductive step and the theorem follows.
\end{proof}

\section{$\mathbb{Z}_2^n$-superdomains and their morphisms}

In view of Theorem \ref{SuffGrad}, we consider in the following only $\Z_2^n$-graded `coordinates' $y^1,\ldots,y^m$, with degree denoted by $\deg(y^i)$, which commute according to the $\Z_2^n$-commutation rule \be\label{Z2nCommRule} y^iy^j=(-1)^{\langle\deg(y^i),\deg(y^j)\rangle}y^jy^i\;,\ee where $\langle-,-\rangle$ is the standard `scalar product' in $\Z_2^n\,.$ The \emph{parity} of $y^i$ is the parity of the sum of the components of $\deg(y^i)\in\Z_2^n$. To develop a generalization of Supergeometry, we distinguish the coordinates $x^1,\dots, x^p$ of degree $0\in\Z_2^n$ and view them as \emph{variables} in an open subset $U\subset\R^p$. The remaining coordinates $\zx^1,\dots,\zx^q$ have nontrivial degrees $\deg(\zx^a)\in\Z_2^n\setminus\{ 0\}$ -- we refer to them as \emph{indeterminates} or \emph{formal variables}.

\subsection{Function sheaf of formal power series}

The first idea is to define the function sheaf $\frak O_U$ of a $\Z_2^n$-superdomain $\frak{U}=(U,\frak{O}_U),$ over any open $V\subset U$, as the $\Z_2^n$-commutative associative unital $\R$-algebra \be\label{local} \frak{O}_U(V)=C^\infty_U(V)[\zx^1,\dots,\zx^q] \ee of polynomials in the indeterminates $\zx^a$ with coefficients in smooth functions of $V$.\medskip

However, this approach has clear shortcomings.\medskip

First, {the} fundamental supergeometric result states that a function of a superdomain is invertible if and only if its projection onto the base is invertible. It is due to the nilpotency of the kernel $\frak J(V)$ of the base projection. However, as we allow here formal variables which are even, the ideal $\frak{J}(V)$ is not nilpotent in our setting and the mentioned basic result fails. As easily understood, we can remedy this problem by substituting polynomials with formal power series.

Second, for a proper development of differential calculus, we must be able to compose elements of degree $0$ with smooth functions. But what is $F(x+\zx^2)$ for a 1-variable smooth function $F$, a variable $x$ and a formal even variable $\zx\,$? Since $\xi$ is not nilpotent, the Taylor formula $F(x + \xi^2) = \sum_k\frac{1}{k!}\,F^{(k)}(x)\, \xi^{2k}$ leads again to a formal power series.

We are thus forced to complete the above structure sheaf to {\it formal power series in the indeterminates $\zx^a$}. This local model is the same than the one obtained by Molotkov \cite{Mol10} via his {\em functorial approach} to higher graded supermanifolds.\medskip

Further motivation for our choice of a {\it base made only of the zero degree coordinates} comes from the following observations. It should be noticed that, when the $\Z_2^n$ sign rule (\ref{Z2nCommRule}) replaces the classical super sign rule, even nonzero degree indeterminates may {\it anticommute} with even and with odd indeterminates (e.g., if they are $\Z_2^3$-graded and have the degrees $(1,1,0)$, $(0,1,1)$ and $(0,1,0)$, respectively). On the other hand, the algebra of quaternions $$\mathbb{H}=\R\oplus \R i\oplus \R j \oplus \R k$$ is $\Z_2^3$-commutative, if we choose the degrees $\deg (1)=(0,0,0)$, $\deg (i)=(1,1,0)$, $\deg (j)=(1,0,1)$, and $\deg (k)=(0,1,1)$. {\it All this shows that the even nonzero degree indeterminates are not usual even variables and that we should think about them as formal variables rather than as ordinary ones}. This confirms that the local $\Z_2^n$-superfunction algebra will be made of formal power series in the odd and the nonzero degree even indeterminates. Indeed, what would for instance be the definition of a smooth function with respect to $i,j,k$? What would be the meaning of the sine $\sin(a+bi+cj+dk)$ of a quaternion?

It follows that in the case of the tangent bundle to a supermanifold $\cM$, the local functions of the supermanifold $\sT[1]\cM$ and of the $\Z_2^2$-supermanifold $\sT\cM$ are quite different. As mentioned in the introduction, if $(x^i,\zx^a)$ are even and odd coordinates of $\cM$, the supermanifold $\sT[1]\cM$ has coordinates $(x^i,\zx^a,\dot x^j,\dot \zx^b)$ of parities $(0,1,1,0)$ and subject to the super commutation rule. The superfunctions of $\sT[1]\cM$ are thus locally the polynomials in $\zx^a,\dot x^j$ with coefficients in the smooth functions with respect to $x^i,\dot \zx^b$. The case of the $\Z_2^2$-supermanifold $\sT\cM$, with coordinates $(x^i,\zx^a,\dot x^j,\dot \zx^b)$ of bidegrees $((0,0),(0,1),(1,0),(1,1))$ and subject to the $\Z_2^2$-commutation rule, is different. Its $\Z_2^2$-superfunctions are locally the formal power series in the nonzero degree formal variables $\zx^a,\dot x^j, \dot \zx^b$ with coefficients in the smooth functions in $x^i$.

It is eventually clear that the base of a $\Z_2^n$-supermanifold corresponds, not to even variables, but to zero-degree ones. In this sense $\Z_2^n$-supermanifolds are {\it similar to $\Z$-graded manifolds}. Indeed, the base of a $\Z$-graded manifold $\cM$ can be recovered as the spectrum of $C^0(\cM)/(I\cap C^0(\cM))$, where $C^0(\cM)$ are the zero-degree functions of $\cM$ and where $I$ is the ideal generated by the coordinates of nonzero degree.\medskip

In what follows, we consider the $2^n$ tuples $s_0,s_1,\ldots,s_{2^n-1}\in\Z_2^n$ as ordered lexicographically: $s_0<s_1<\ldots<s_{2^n-1}\,$.

\begin{definition} Let $n,p,q_1,\ldots,q_{2^n-1}\in\N$ and set $\mathbf{q}=(q_1,\ldots,q_{2^n-1})$. Consider $p$ coordinates $x^1,\ldots,x^p$ of degree $s_0=0$ (resp., $q_1$ coordinates $\xi^1,\ldots,\xi^{q_1}$ of degree $s_1$, $q_2$ coordinates $\xi^{q_1+1},\ldots,\xi^{q_1+q_2}$ of degree $s_2$, ...) and denote by $x=(x^1,\ldots,x^p)$ (resp., $\xi=(\xi^1,\ldots,\xi^q)$) the tuple of all zero degree (resp., all the nonzero degree) coordinates (of course $q=\sum_kq_k$). These coordinates $(x,\xi)$ commute according to the $\Z_2^n$-commutation rule (\ref{Z2nCommRule}). A \emph{$\Z_2^n$-superdomain} (called also a \emph{color superdomain}) of dimension $p|\mathbf{q}$ is a \emph{ringed space} ${\cU^{\,p|\mathbf{q}}} = (U,{\cO}_{U})$, where $U\subset\R^p$ is the open range of $x$, and where the structure sheaf is defined over any open $V\subset U$ as the $\Z_2^n$-commutative associative unital $\R$-algebra
\be\label{local}
{\cO}_U(V)=C^\infty_U(V)[[\zx^1,\dots,\zx^q]]
\ee
of formal power series
\be\label{pf1}
 f(x,\zx)= \sum_{|\mu|=0}^{\infty}   f_{\mu_1\ldots\mu_q}(x)\;(\zx^1)^{\mu_1}\ldots(\zx^q)^{\mu_q} =\sum_{|\mu|=0}^{\infty} f_\mu(x) \zx^\mu\;
\ee
in the formal variables $\xi^1,\dots,\xi^q$ with coefficients in $\Ci_U(V)$ (standard multiindex notation).\end{definition}

We refer to any ringed space of $\Z_2^n$-commutative associative unital $\R$-algebras as a \emph{$\Z_2^n$-ringed space} and to the preceding functions (\ref{pf1}) as the {\em local $\Z_2^n$-superfunctions} (or just $\Z_2^n$-functions).

\begin{example}\label{FundaExa} Consider the case $n=2$ and $p|(q_1,q_2,q_3)=1|(1,1,1)$, write for simplicity $(x,\xi,\zh,\zy)$ instead of $(x,\xi^1,\xi^2,\xi^3)$, and note that the degrees of these coordinates are $(0,0)$, $(0,1)$, $(1,0)$, and $(1,1),$ respectively. A $\Z_2^2$-function is then of the form $$\label{Superfunction}f(x,\xi,\zh,\zy)=\sum_{r\ge 0} f_r(x)\zy^{2r}+\sum_{r\ge 0} g_r(x)\zy^{2r+1}\xi\zh+\sum_{r\ge 0} h_r(x)\zy^{2r}\xi+\sum_{r\ge 0} k_r(x)\zy^{2r+1}\zh$$ \be +\sum_{r\ge 0} \ell_r(x)\zy^{2r}\zh+\sum_{r\ge 0} m_r(x)\zy^{2r+1}\xi+\sum_{r\ge 0} n_r(x)\zy^{2r+1}+\sum_{r\ge 0} p_r(x)\zy^{2r}\xi\zh\;,\ee where the sums are formal series and the functions in $x$ are smooth. Note that the first (resp., second, third, fourth) two sums contain terms of $\Z_2^2$-degree $(0,0)$ (resp., $(0,1)$, $(1,0)$, and $(1,1)$).\end{example}

\subsection{Locality of $\Z_2^n$-superdomains}

In classical Supergeometry, a (super) ringed space is called a \emph{space} if all its stalks are local rings, i.e., rings that have a unique maximal homogeneous ideal. Such ringed spaces are referred to as \emph{locally ringed spaces}. Further, a locally ringed space is a \emph{supermanifold} if it is locally modelled on a superdomain. Superdomains are therefore `trivial' locally ringed spaces. Of course, one has to verify that the stalks of a superdomain are local rings.\medskip

To show that the stalks of a $\Z_2^n$-superdomain are local rings, i.e., that these domains are \emph{locally $\Z_2 ^n$-ringed spaces}, we need two well-known lemmas, where the second is a consequence of the first.

Consider the algebra (\ref{local}), or, a bit more generally, the $\Z_2^n$-commutative associative unital $R$-algebra $R[[\xi^1,\ldots,\xi^q]]$, over a commutative unital ring $R$, whose elements are assumed to be central.

\begin{lem} A series $1-v$ in $R[[\xi^1,\ldots,\xi^q]]$, where $v$ has no independent term, is invertible, with inverse $v^{-1}=\sum_{k\geq0} v^k$.\end{lem}

\begin{lem}\label{lem:FPSinvert}
A series in $R[[\xi^1,\ldots,\xi^q]]$ is invertible if and only if its independent term is invertible in $R$.
\end{lem}

\begin{prop}\label{prop:domlocality} Any $\Z_2^n$-superdomain $(U,C^\infty_U[[\xi^1,\dots,\xi^q]])$ is a locally $\Z_2^n$-ringed space, i.e., for any $x\in U$, the stalk $\Ci_{U,x}[[\xi^1,\ldots,\xi^q]]$ has a unique maximal homogeneous ideal
\be\label{MaxIdSupDom}
\frak{m}_x=\{[f]_x: f_0(x)=0\}\;,
\ee where $f_0$ is the independent term of $f$.
\end{prop}

Lemma \ref{lem:FPSinvert} implies not only Proposition \ref{prop:domlocality} but also the following.

\begin{cor}\label{p1} For any open $V\subset U$, a $\Z_2^n$-function $$f\in {\cO}_U(V)=\Ci_U(V)[[\xi^1,\ldots,\xi^q]]$$ is invertible in ${\cO}_U(V)$ if and only if its independent term $f_0$ is invertible in $\Ci_U(V)$.
\end{cor}

This corollary guarantees that crucial results of classical Supergeometry still hold in $\Z_2^n$-Supergeometry, although formal variables are no longer necessarily nilpotent. Among these consequences are the existence of a smooth structure on the topological base space $M$ of any $\Z_2^n$-supermanifold $(M,\cA_M)$, as well as the fundamental short exact sequence of sheaves $0\to\cJ_M\to\cA_M\to\cC^\infty_M\to 0\,.$

\subsection{Completeness of $\Z_2^n$-function algebras}\label{ComplDom}

Consider a $\Z_2^n$-superdomain with $\Z_2^n$-functions $$f(x,\xi)=\sum_{|\mu|=0}^{\infty} f_\mu(x) \zx^\mu\in \cO=\Ci[[\xi^1,\ldots,\xi^q]]\;,$$ where we omitted $U$ and $V$ for convenience. The number $k:=|\zm|$ of generators defines an $\N$-grading in $\cO$ that induces a decreasing filtration $\cO_\ell=\Ci[[\xi^1,\ldots,\xi^q]]_{\ge\ell}$, where subscript $\ge \ell$ means that we consider only those series whose terms contain at least $\ell$ parameters $\zx^a$. If $J=\ker\ze$ denotes the kernel of the projection $\ze:\cO\ni f\mapsto f_0\in\Ci$ of $\Z_2^n$- onto base-functions, the ideal $J$ is given by $J=J^1=\cO_1$ and $J^\ell=\cO_\ell\,$: $\cO\supset J\supset J^2\supset\ldots$ The sequence $\cO/J\leftarrow \cO/J^2\leftarrow\cO/J^3\leftarrow\ldots\;$, which can be identified with the sequence $\Ci\leftarrow \Ci[[\xi^1,\ldots,\xi^q]]_{\le 1}\leftarrow \Ci[[\xi^1,\ldots,\xi^q]]_{\le 2}\leftarrow\ldots\;$, is an inverse system, whose limit is \be\label{Completeness}\varprojlim_\ell \cO/J^\ell=\cO\;.\ee Moreover, if $\frak J$ denotes the similar ideal defined in the frame of polynomial $\Z_2^n$-functions $\frak O$, see above, we also have \be\label{Completeness2}\varprojlim_\ell \frak{O}/\frak{J}^\ell=\cO\;.\ee

\begin{prop}\label{CompletnessLoc} The algebra $\cO_U(V)=\Ci_U(V)[[\xi^1,\ldots,\xi^q]]$ of local $\,\Z_2^n$-functions is Hausdorff-complete with respect to the $J_U(V)$-adic topology, and this algebra $\cO_U(V)$ of formal power series is the completion $\widehat{\frak O}_U(V)$ with respect to the $\frak J\,_U(V)$-adic topology of the algebra $\frak O_U(V)$ of polynomials.\end{prop}

Proposition \ref{CompletnessLoc} will be used in the proof of the fundamental morphism theorem for $\Z_2^n$-supermanifolds, which we sketch in the next subsection in the case of $\Z_2^n$-superdomains.

\subsection{Morphisms of $\Z_2^n$-superdomains}\label{MorphTheo}

The following remark shows that {\it morphisms of $\Z_2^n$-superdomains} can be viewed as in classical differential Geometry. It will be formulated more rigorously in Section \ref{mor}. \medskip

Consider two $\Z_2^n$-superdomains of dimension $p|\mathbf{q}$ and $p'|\mathbf{q}'$ over open subsets $U\subset\R^p$ and $U'\subset\R^{p'}$, respectively. Roughly, $\Z_2^n$-morphisms between these $\Z_2^n$-superdomains correspond to {\it graded unital $\R$-algebra morphisms}
$$\phi^*:\Ci(V')[[\zx'^1,\dots,\zx'^{q'}]]\to \Ci(V)[[\zx^1,\dots,\zx^{q}]]\;$$
and are {\it determined by their coordinate form}
\bea\label{mor1}
x'^i&=&\zf^i(x)+\sum_{\zs(\mu)=0} f^i_\mu(x) \zx^\mu\,,\\
\zx'^a&=&\sum_{\zs(\mu)=\deg(\zx'^a)} f^a_\mu(x) \zx^\mu\,,\nn
\eea
where the sums are formal series with coefficients in smooth functions, where $\zs(\zm)$ is the degree $\zm_1\deg(\zx^1)+\ldots+\zm_q\deg(\zx^q)$ of the term characterized by $\zm$, and where $\zf:V\ni (x^1,\ldots,x^p)\mapsto (x'^1,\ldots,x'^{p'})\in V'$ is a smooth map.\medskip

\begin{example}\label{FundaExa2} In the case of $\Z_2^2$-superdomains of dimension $1|(1,1,1)$ with variables $(x,\xi,\zh,\zy)$ (resp., $(y,\za,\zb,\zg)$) of $\Z_2^2$-degrees $((0,0),(0,1),(1,0),(1,1))$, a $\Z_2^2$-morphism can be viewed as usual: \be\label{Morph}\left\{\begin{array}{c} y=\sum_r f^y_r(x)\zy^{2r}+\sum_r g^y_r(x)\zy^{2r+1}\xi\zh\;,\\ \za=\sum_r f^\za_r(x)\zy^{2r}\xi+\sum_r g^\za_r(x)\zy^{2r+1}\zh\;, \\ \zb=\sum_r f^\zb_r(x)\zy^{2r}\zh+\sum_r g^\zb_r(x)\zy^{2r+1}\xi\;, \\ \zg=\sum_r f^\zg_r(x)\zy^{2r+1}+\sum_r g^\zg_r(x)\zy^{2r}\xi\zh\;.\end{array}\right.\;\ee\end{example}

\begin{remark} It is easily seen from Equation (\ref{Morph}) that the Jacobian matrix of this $\Z_2^n$-morphism is a $\Z_2^n$-degree $0$ matrix with entries in the $\Z_2^n$-function algebra. In the case of a $\Z_2^n$-diffeomorphism, this matrix is invertible, so that we can compute its $\Z_2^n$-Berezinian \cite{COP12} and further develop the $\Z_2^n$ integral calculus (which is actually the core of the present $\Z_2^n$-project, since it turned out to be really different from ordinary super integral calculus). \end{remark}

To explain the above claim about $\Z_2^n$-morphisms, we have to prove that any $\Z_2^n$-morphism has a coordinate form of the announced type (what is almost obvious), and that, conversely, any pullbacks $\phi^*(x'^i)\;(\simeq x'^i)$ and $\phi^*(\xi'^a)\;(\simeq \xi'^a)$ of the form (\ref{mor1}) uniquely extend to a $\Z_2^n$-morphism. We will show here that such a prolonging $\Z_2^n$-morphism does exist. Uniqueness (and other details) will be proven independently in the more general case of $\Z_2^n$-morphisms of $\Z_2^n$-supermanifolds.\medskip

In the sequel we write $\phi^*(x'^ i)=\zf^i(x)+j^i(x,\xi)$, with $j^i(x,\zx)=\sum_{\zs(\mu)=0} f^i_\mu(x) \zx^\mu\in J$. For any $$g(x',\xi')=\sum_{|\zn|\ge 0}g_\zn(x')\xi'^\zn\in \Ci(V')[[\zx'^1,\dots,\zx'^{q'}]]\;,$$ we set \be\label{Pullback}(\phi^*(g))(x,\xi)=\sum_{|\zn|\ge 0} \phi^*(g_\zn(x')) (\zvf^*(\xi'))^\zn\;,\ee where
\be\label{FormTayl} \phi^*(g_\zn(x'))=g_\zn(\phi^*(x'))=g_\zn(\zf(x)+j(x,\zx))= \sum_{|\za|\ge 0} \frac{1}{\za!}\; (\p ^{\za}_{x'}g_{\zn})(\zf(x))\; j^{\za}(x,\xi)\;\ee is a formal Taylor expansion -- we use here the multiindex notation: $j^\za=(j^1)^{\za^1}\ldots (j^{p'})^{\za^{p'}}\in J^{|\za|}$. In fact the {\small RHS} of (\ref{FormTayl}) is a series of series and it could lead to a rearranged series with non-converging series of $\Ci(V)$-coefficients. However, any type of monomial in the formal variables $\xi^a$ contains a fixed number $N$ of parameters. As the terms indexed by $|\za|>N$ contain at least $N+1$ parameters, they not contribute to the considered monomial. The coefficient of the latter is therefore a finite sum in $\Ci(V)$, so that the {\small RHS} of (\ref{FormTayl}) is actually a series in $\Ci(V)[[\zx^1,\dots,\zx^{q}]]\,$. The same argument can be used for the {\small RHS} of (\ref{Pullback}).\medskip

It is easily seen that the thus defined pullback map $\phi^*$ is a graded unital $\R$-algebra morphism.\medskip

As mentioned before, the precise definition of a $\Z_2^n$-morphism as a morphism of locally $\Z_2^n$-ringed spaces will be given in Section $\ref{Z2nSuperMan}$, and the preceding explanation will be completed and generalized.

\section{$\Z_2^n$-Supermanifolds}\label{Z2nSuperMan}

\subsection{Definitions}\label{Definition}

It is clear that a $\Z_2^n$-supermanifold is locally modelled on a $\Z_2^n$-superdomain. Here `modelled' means `isomorphic' as locally $\Z_2^n$-ringed space. Such an isomorphism is of course an invertible morphism of locally $\Z_2^n$-ringed spaces. The morphisms of the category $\tt LZRS$ of locally $\Z_2^n$-ringed spaces are the morphisms of $\Z_2^n$-ringed spaces, i.e. the sheaf morphisms $\Psi=(\psi,\psi^*)$ such that the naturally induced algebra morphisms between stalks respect the maximal ideals. In the following, all topological spaces are assumed to be Hausdorff and second-countable and the elements of $\Z_2^n$, $n\in\N$, are considered as ordered lexicographically.

\begin{definition}[Ringed space definition]\label{DefZSupMan} A (smooth) \emph{$\Z_2^n$-supermanifold} (or a \emph{color supermanifold}) $\cM$ of dimension $p|\mathbf{q}$, $p\in\N$, $\mathbf{q}=(q_1,\ldots,q_{2^n-1})\in \N^{2^n-1}$, is a locally $\Z_2^n$-ringed space $(M,\cA_M)$ that is locally isomorphic to the $\Z_2^n$-superdomain $(\R^p,\Ci_{\R^p}[[\xi^1,\ldots,\xi^q]])$, where $q=\sum_kq_k$, where $\xi^1,\ldots,\xi^q$ are $\Z_2^n$-commuting formal variables of which $q_k$ have the $k$-th degree in $\Z_2^n\setminus\{0\}$, and where $\Ci_{\R^p}$ is the function sheaf of the Euclidean space $\R^p$. \end{definition}

Many geometric concepts can be glued from local pieces: {\it they can be defined via a cover by coordinate systems, with specific coordinate transformations that satisfy the usual cocycle condition}. The same holds for $\Z_2^n$-supermanifolds.

\begin{remark} Roughly, a $\Z_2^n$-supermanifold can be viewed as a topological space $M$, which is covered by {\it $\Z_2^n$-graded $\Z_2^n$-commutative coordinate systems} $(x,\xi)$ ($x$ can be interpreted as a homeomorphism $x(m)\rightleftarrows m(x)$ between its Euclidean open range $U$ and an open subset of $M$ (which is often also denoted by $U$)) and is endowed with {\it coordinate transformations that respect the $\Z_2^n$-degree} and satisfy the {\it cocycle condition}.\end{remark}

The precise alternative definition of $\Z_2^n$-supermanifolds follows naturally from this approximate idea.

\begin{definition} A \emph{$\Z_2^n$-chart} (or $\Z_2^n$-coordinate system) over a topological space $M$ is a locally $\Z_2^n$-ringed space $${\cal U}=(U,\Ci_{U}[[\xi^1,\ldots,\xi^q]]),\; U\subset\R^p, p,q\in\N\;,$$ together with a homeomorphism $\psi:U\to \psi(U)$, where $\psi(U)$ is an open subset of $M$.\end{definition}

Given two {$\Z_2^n$-charts} $({\cal U}_\za,\psi_\za)$ and $({\cal U}_\zb,\psi_\zb)$ over $M$, we will denote by $\psi_{\zb\za}$ the homeomorphism \be\label{UnderHomeo}\psi_{\zb\za}:=\psi_\zb^{-1}\psi_\za:V_{\zb\za}:=\psi_\za^{-1}(\psi_\za(U_\za)\cap \psi_\zb(U_\zb))\to V_{\za\zb}:=\psi_\zb^{-1}(\psi_\za(U_\za)\cap \psi_\zb(U_\zb))\;.\ee

Whereas in classical Differential Geometry the coordinate transformations are completely defined by the coordinate systems, in $\Z_2^n$-Supergeometry, they have to be specified separately.\medskip

\begin{definition} A \emph{coordinate transformation} between two $\Z_2^n$-charts $({\cal U}_\za,\psi_\za)$ and $({\cal U}_\zb,\psi_\zb)$ over $M$ is an {isomorphism} of locally $\Z_2^n$-ringed space $\Psi_{\zb\za}=(\psi_{\zb\za},\psi^*_{\zb\za}):{\cal U}_{\za}|_{V_{\zb\za}}\to {\cal U}_{\zb}|_{V_{\za\zb}},$ where the source and target are restrictions of `sheaves' (note that the underlying homeomorphism is (\ref{UnderHomeo})).

A $\Z_2^n$-\emph{atlas} over $M$ is a covering $({\cal U}_\za,\psi_\za)_\za$ by $\Z_2^n$-charts together with a coordinate transformation for each pair of $\Z_2^n$-charts, such that the usual {cocycle} condition $\Psi_{\zb\zg}\Psi_{\zg\za}=\Psi_{\zb\za}$ holds (appropriate restrictions are understood).\end{definition}

\begin{definition}[Atlas definition] A (smooth) \emph{$\Z_2^n$-supermanifold} $\cM$ is a topological space $M$ together with a preferred $\Z_2^n$-atlas $({\cal U}_\za,\psi_\za)_\za$ over it. \end{definition}

\section{Examples}

\begin{example} For $n=1$, we recover classical supermanifolds. Indeed, in this case there are no formal variables that bear powers higher than 1 and formal series are thus just polynomials.
\end{example}

\begin{example}\label{TB} We already mentioned that the tangent bundle $\sss T \cM$ of a standard $\Z_2$-super\-man\-i\-fold $\cM$ gives rise to a $\Z_2^2$-supermanifold. Indeed, the local coordinates $(x,\zx)$ on $\cM$ induce canonically local coordinates $(x,\zx,\dot x,\dot \zx)$ on ${\sss T \cM}$, which transform according to the rule \be\label{CoordTransTB}\begin{array}{l} x'=x'(x,\zx)\\ \zx'=\zx'(x,\zx)\end{array}\quad \text{and}\quad\begin{array}{l}\dot x'=\dot x\,\p_{x}x'+\dot\zx\,\p_\zx x'\\\dot\zx'=\dot x\,\p_x\zx'+\dot\zx\,\p_\zx\zx'\end{array}\;.\ee We now assign the $\Z_2^2$-degrees $((0,0),(0,1),(1,0),(1,1))$ to the coordinates $(x,\zx,\dot x,\dot \zx)$ and consider them as $\Z_2^2$-commutative. In view of (\ref{CoordTransTB}), this can be done in a coherent manner. The cocycle condition reduces to the chain rule. Hence, the tangent bundle $\sss T\cM$ is a $\Z_2^2$-supermanifold. Moreover, again due to (\ref{CoordTransTB}), the algebra of $\Z_2^2$-superfunctions which are polynomial in the formal variables is well-defined. It can be identified with the algebra $\zW_D({\cM})$ of Deligne super differential forms on $\cM$. Since it is dense in the whole algebra of $\Z_2^2$-superfunctions of $\sss T\cM$, the latter can be identified with the corresponding completion $$\widehat{\zW}_{\op{D}}(\cM)=\prod_{k\ge 0}\w^k\zW^1_{\op{D}}(\cM)\,.$$

\begin{prop} The tangent bundle $\sss T\cM$ of a $\Z_2$-supermanifold $\cM=(M,\cA_M)$, interpreted as ringed space $(M,\widehat{\zW}_D(\cM))$, is a $\Z_2^2$-supermanifold.
\end{prop}

\end{example}

\section{$\Z_2^n$-superizations of $n$-fold vector bundles}
Let now
\begin{equation}\label{algebroid}
        {\xymatrix@R-1mm @C-1mm{ & E \ar[ld]_*{\zt_l} \ar[rd]^*{\zt_r} & \cr
        E_{10} \ar[rd]_*{\bar\zt_r} & E_{11} \ar[u] & E_{01} \ar[ld]^*{\bar\zt_l}  \cr & M  & }}
    \end{equation}
be a double vector bundle. This corresponds to a choice of two commuting Euler vector fields on $E$ \cite{GR09}. Alternatively, one can interpret a double vector bundle also as a locally trivial fiber bundle (not vector bundle, see (\ref{CoordTransDB})),
whose standard fiber is a graded vector space $V_{01}\oplus V_{10}\oplus V_{11}$, and whose coordinate
transformations are of the type
\bea\label{CoordTransDB}
x'&=&\zf(x)\,,\nonumber\\
\zx'&=&a(x)\zx\,,\nonumber\\
\zh'&=&b(x)\zh\,,\nonumber\\
\zy'&=&c(x)\zy+d(x)\zx\zh,
\eea
where $\zx,\zx'$ (resp., $\zh,\zh'$ and $\zy,\zy'$) are coordinates in $V_{01}$ (resp., $V_{10}$ and $V_{11}$), and where $x,x'$ are coordinates in the base \cite{Vor12}.
We now $\Z_2^2$-superize, assigning the $\Z_2^2$-degrees $((0,0),(0,1),(1,0),$ $(1,1))$ to the coordinates $(x,\zx,\zh,\zy)$ and $(x',\zx',\zh',\zy')$.

Notice that the coordinate changes respect the $\Z_2^2$-degree.
The coordinates $\zx$ and $\zh$ are odd, so when trying to obtain a superization into a standard supermanifold, we have to decide the order in the product $\zx\zh$, which is irrelevant for the (even) double vector bundle. Moreover, we should agree these orders so that
the cocycle condition is still satisfied. This can be done \cite{GR09,Vor12} for the price
of fixing some orders for consecutive parity changes with respect to all the vector bundle structures (see discussion in \cite{BGR15}).

It can easily be seen that the problem cannot arise in the present situation. Indeed, in the $\Z_2^2$-superization the coordinates $\zx$ and $\zh$ are odd with bidegrees $(0,1)$ and $(1,0)$, respectively, so they \textbf{commute} and the problem of ordering disappears.

All this can be repeated in the general case of an $n$-fold vector bundle $E$. Indeed, there is an $\N^n$-gradation in a dense subsheaf of the structure sheaf of $E$, given by the $n$ Euler vector fields, and an atlas in which all local coordinates are homogeneous and have degrees $\le \ao$ with respect to the lexicographical order, where $\ao=(1,\dots, 1)\in\{ 0,1\}^n$. As changes of coordinates respect the multidegree, all coordinate  changes are polynomial in nonzero degrees and all products which appear in these polynomials are products of coordinates with degrees having disjoint supports (otherwise  we would get degrees with some entries $>1$), therefore commuting in the $\Z_2^n$-setting (cf. \cite{GR09}). The problem of ordering disappears and the coordinate changes in the even $n$-fold vector bundle $E$ remain valid and consistent also for its $\Z_2^n$-superization. In other words, the degrees of local coordinates and their transformation laws in $E$ remain consistent also with the $\Z_2^n$-sign rules (\ref{Z2nCommRule}).
Thus we get the following.

\begin{prop} There exist a canonical $\Z_2^n$-superization of any $n$-fold vector bundle $E$ which gives a $\Z_2^n$-supermanifold $\zP E$ with the same homogeneous local coordinates, their degrees in $\Z_2^n$ and their transformations laws, but satisfying the new sign rules (\ref{Z2nCommRule}).\end{prop}

\section{Morphisms of $\Z_2^n$-supermanifolds}\label{mor}

\begin{remark} In contrast with integral calculus, most of the classical super differential calculus goes through in the $\Z_2^n$-graded context. Below we give the results along with some few explanations. The detailed proofs in the $\Z_2^n$-graded setting can be found in the preprint \cite{CGP1}.\end{remark}

\subsection{Embedding of the smooth base manifold}

\begin{prop} The base topological space $M$ of any $\Z_2^n$-supermanifold $\cM=(M,\cA_M)$ of dimension $p|\mathbf{q}$ carries a smooth manifold structure of dimension $p$, and there exists a short exact sequence of sheaves of $\Z_2^n$-commutative associative $\R$-algebras $$0\to {\cal J}_M\to \cA_M\stackrel{\ze}{\to }\Ci_M\to 0\;.$$\end{prop}

\begin{proof}[Sketch of proof] In view of Corollary \ref{p1}, there exists, for any open $U\subset M$, for any $f\in\cA_M(U)$ and any $m\in U$, a unique $k\in\R$ such that $f-k$ is not invertible, in any neighborhood of $m$. We set \be\label{DefBaseProj}\ze_U(f)(m)=k\ee and $\cJ_M=\ker\ze$. The presheaf $\op{im}\ze$ generates a sheaf $\frak{F}_M$ which is locally isomorphic to $\Ci_{\R^p}$ and thus implements a $p$-dimensional smooth manifold structure on $M$ such that $\Ci_M\simeq \frak{F}_M$.\end{proof}

The description of the basis of neighborhoods of 0 in the $\cJ_M(U)$-adic topology of $\cA_M(U)$ in a $\Z_2^n$ coordinate domain $U$ is clear from Subsection \ref{ComplDom}.

\subsection{Continuity of morphisms}\label{ContMorph}
\newcommand{\cB}{{\cal B}}

\begin{definition} A \emph{morphism of $\Z_2^n$-supermanifolds} or \emph{$\Z_2^n$-morphism} is a morphism of the underlying locally $\Z_2^n$-ringed spaces.\end{definition}

The next commutation property is fundamental. It is a consequence of the definition (\ref{DefBaseProj}).

\begin{prop} \label{commepspsi}
Let $$ \Psi=(\psi,\psi^*) : {\cal M} =(M,\cA_M)  \to  {\cal N} =(N,\cB_N) $$ be a morphism of $\Z_2^n$-supermanifolds, let $V\subset N$ be open, and $U=\psi^{-1}(V)$. Then,
\be\label{ComMorpBaseProj} \ze_{U} \circ \psi^{*}_V = \psi^{*}_V \circ \ze_{V}\;, \ee where the {\small LHS} pullback of $\Z_2^n$-functions is given by the second component of $\Psi$ and where the {\small RHS} pullback of classical functions is equal to $- \circ \psi|_U $ and thus given by the first component of $\Psi$.\end{prop}

\begin{cor} For any $\Z_2^n$-supermanifold ${\cal M} =(M,\cA_M)$ and any point $m\in M$, the unique maximal homogeneous ideal ${\frak m}_m$ of the stalk $\cA_m$ is given by \be\label{MaxIdea}{\frak m}_m=\{[f]_m:(\ze f)(m)=0\}\;.\ee\end{cor}

The result is easily proven. However, actually we need not assume that a $\Z_2^n$-supermanifold is a locally $\Z_2^n$-ringed space. If we omit this requirement, the local isomorphism $\zF=(\zvf,\zvf^*)$ from the $\Z_2^n$-supermanifold to its local model is only an isomorphism in $\Z_2^n$-ringed spaces. Nevertheless, since the induced morphisms $\zvf^*_m$ between stalks are isomorphisms of graded unital $\R$-algebras, $\zvf^*_m({\frak m}_x)$, where $x=\zvf(m)$ and $\frak m_x$ is the unique maximal homogeneous ideal (\ref{MaxIdSupDom}) of the stalk at $x$ of the local model, \emph{is} the unique maximal homogeneous ideal of $\cA_m$. It then follows from Proposition \ref{commepspsi} that it is given by (\ref{MaxIdea}).\medskip

We can choose a centered chart $(x,\xi)$ around $m$ and work in a $\Z_2^n$-superdomain over a convex open subset. In view of (\ref{MaxIdea}), a Taylor expansion (with remainder) around $m\simeq x=0$ of the coordinate form of $\ze f$ shows that $${\frak m}_m\simeq \{[f]_0: f(x,\xi) = 0(x) + \sum_{|\zm|>0}f_\zm(x)\xi^\zm \}\;,$$ where $0(x)$ are terms of degree at least 1 in $x$. More generally:

\begin{lem}\label{Claim4} For any $m\in M$, the basis of neighborhoods of $\,0$ in the ${\frak m}_m$-adic topology of $\cA_m$ is given by \be{\frak m}_m^{k+1}=\{[f]_0: f(x,\xi) = \sum_{0\le|\zm|\le k}0_\zm(x^{k-|\zm|+1})\xi^\zm+\sum_{|\zm|>k}f_\zm(x)\xi^\zm\}\quad(k\ge 0)\;,\label{lc2}\ee where notation is the same as above. \end{lem}

The next result is a consequence of the definition $\cJ=\ker\ze$, Equation (\ref{MaxIdea}), and Proposition \ref{commepspsi}. It shows in particular that $\Z_2^n$-morphisms automatically respect maximal ideals, so that this requirement is actually redundant in their definition.

\begin{cor}\label{Claim1} Any $\Z_2^n$-morphism $\Psi=(\psi,\psi^*):\cM=(M,\cA_M)\to \cN=(N,\cB_N)$ is continuous with respect to $\cal J$ and $\frak m$, i.e., for any open $V\subset N$ and any $m\in M$, we have
$$ \psi^{*}_V\lp \cJ_N(V) \rp \subset \cJ_M({\psi^{-1}(V)})\;\;\text{and}\;\; \psi^{*}_{m}\lp \frak{m}_{\psi(m)}\rp \subset \frak{m}_{m}\;.\label{stalkcond}$$
\end{cor}

\begin{cor} The base map $\psi:M\to N$ of any $\Z_2^n$-morphism $\Psi:(M,\cA_M)\to (N,\cB_N)$ is smooth.\end{cor}

\subsection{Completeness of the $\Z_2^n$-function sheaf and the $\Z_2^n$-function algebras}

The decreasing filtration $\cA\supset \cJ\supset \cJ^2\supset\ldots$ of the structure sheaf of a $\Z_2^n$-supermanifold $\cM=(M,\cA)$ gives rise to an inverse system $$\cA/\cJ\leftarrow \cA/\cJ^2\leftarrow \cA/\cJ^3\leftarrow\ldots $$ of sheaves of algebras (the quotient presheaves $\cA/\cJ^k$ are actually sheaves due to the existence of partitions of unity on $\cM$). Since a limit is a universal cone, there exists a sheaf morphism $\varprojlim_k\cA/\cJ^k\leftarrow \cA$. Moreover, as a limit in a category of sheaves is just the corresponding limit in the category of presheaves (which is computed objectwise), we get, for any $\Z_2^n$ chart domain $U_\za$, $$\lp \varprojlim_k\cA/\cJ^k \rp (U_\za)=\varprojlim_k\cA(U_\za)/\cJ^k(U_\za)\simeq \cA(U_\za)\;,$$ see Equation (\ref{Completeness}). It follows that $$\varprojlim_k\cA/\cJ^k\simeq\cA\;$$ in the category of sheaves. 
Since isomorphic sheaves have isomorphic sections, we thus obtain, for any $U\subset M$, $$\varprojlim_k\cA(U)/\cJ^k(U)\simeq\cA(U)\;.$$

\begin{prop}\label{Hausdorff} The $\Z_2^n$-function sheaf $\cA_M$ (resp., the $\Z_2^n$-function algebra $\cA_M(U),$ $U\subset M$) of a $\Z_2^n$-supermanifold $(M,\cA_M)$ is Hausdorff-complete with respect to the $\cJ_M$-adic (resp., $\cJ_M(U)$-adic) topology.\end{prop}

\subsection{Fundamental theorem of $\Z_2^n$-morphisms}

\begin{thm}\label{FundaTheoSuperm} If ${\cal M}=(M,\cA_M)$ is a $\Z_2^n$-supermanifold of dimension $p|\mathbf{q}\,$, $${\cal V}^{\,u|\mathbf{v}}=(V,\Ci_V[[\zh^1,\ldots,\zh^v]])$$ a $\Z_2^n$-superdomain of dimension $u|\mathbf{v}$, $V\subset\R^u$, $v=|\mathbf{v}|$, and if $(s^j,\zeta^b)$ is an $(u+v)$-tuple of homogeneous $\Z_2^n$-functions in $\cA_M(M)$ that have the same $\Z_2^n$-degrees as the coordinates $(y^j,\zh^b)$ in ${\cal V}^{\,u|\mathbf{v}}$ and satisfy $\lp  \ze s^{1}, \ldots, \ze s^{u}\rp(M)\subset V,$ then there exists a unique morphism of $\Z_2^n$-supermanifolds
$\Psi =(\psi, \psi^{*}) :  {\cal M} \raa {\cal V}^{u|\mathbf{v}},$ such that \be\label{PullCoordFunc} s^{j}= \psi^{*}_Vy^{j}\quad\text{and}\quad \zeta^{b}= \psi^{*}_V\zh^{b}\;.\ee\end{thm}

Hence, $\Z_2^n$-morphisms whose target is a $\Z_2^n$-superdomain are completely defined by their values on the coordinate functions. This fundamental result makes $\Z_2^n$-Supergeometry a reasonable theory.\medskip

Roughly speaking, since $\psi^*_{V}$ is an algebra morphism, the data (\ref{PullCoordFunc}) uniquely determine the pullback $\psi^{*}_VP$ of any section $P\in \op{Pol}_V(V)[[\zh^1,\ldots,\zh^v]]$ with polynomial coefficients. Hence the quest for an appropriate polynomial approximation of an arbitrary section.

\begin{thm}\label{Claim3}
Let $m\in M$ be a base point of a $\Z_2^n$-supermanifold $(M,\cA_M)$ and let $f\in\cA(U)$ be a $\Z_2^n$-function defined in a neighborhood $U$ of $m$. For any fixed degree of approximation $k\in\N\setminus\{0\}$, there exists a polynomial $P=P(y,\zh)$ such that $$[f]_{m}-[P]_{m} \in \frak{m}_{m}^{k}\;.$$
\end{thm}

In this statement, the polynomial $P$ depends on $m,$ $f$, and $k$, and the variables $(y,\zh)$ are (pullbacks of) coordinates centered at $m$. It is a direct consequence of a Taylor expansion at $m=y=0$ of the coefficients of $f=f(y,\zh)$.

\begin{proof}[Sketch of proof of Theorem \ref{FundaTheoSuperm}] Since the existence of an extending $\Z_2^n$-morphism can be reduced in the main to the construction described in Section \ref{MorphTheo}, we focus here on its uniqueness.

Let $\Psi_1=(\psi_{1},\psi_{1}^{*})$ and $\Psi_2=(\psi_{2},\psi_{2}^{*})$ be two $\Z_2^n$-morphisms defined by the same $(s^j,\zeta^b)$. First, the commutation property (\ref{ComMorpBaseProj}) allows to see that $\psi_1=\psi_2$. As for $\psi_1^*$ and $\psi_2^*$, one shows (rigorously) that they coincide on polynomial sections (using the continuity \ref{Claim1} of the pullbacks of sections with respect to the $\cJ$-adic topology, as well as the Hausdorff-completeness \ref{Hausdorff} of the $\Z_2^n$-function algebras), then one proves that they coincide on arbitrary sections (using the polynomial approximation \ref{Claim3}, the continuity \ref{Claim1} of the pullbacks of germs with respect to the ${\frak m}$-adic topology, as well as the description \ref{Claim4} of the zero neighborhoods of the latter).\end{proof}

\section{Concluding remarks}
The idea of considering supercommutation rules not reduced to `bosons' and `fermions', but including \emph{a priori} arbitrary statistics for commutation of elements of finite multi-particle states, is not new, but has not been exploited very extensively in the literature. In particular, some attempts to develop a theory of \emph{color supermanifolds} are generally suffering from a lack of interesting examples.\medskip

In this paper we proposed a well-founded geometric theory of this kind with a clear
local model. Our theory is not just a trivial extension of that of supermanifolds. The existence of even formal generators leads to superalgebras containing a Grassmann part, as well as a part consisting of formal power series. Also the corresponding Berezinian determinant and volume are far from being obvious concepts.\medskip

On the other hand, our formalism opens new geometric perspectives and has potential applications.
One can consider $\Z_2^n$-graded Poisson brackets and the corresponding deformation quantization.
Since Clifford algebras can be viewed as deformations of Grassmann algebras induced by an odd Poisson bracket, this includes `higher' or `color' Clifford algebras, as well as `color' Moyal brackets. More precisely, our \emph{color Clifford algebra} is generated as the free (tensor) algebra of a $\Z_2^n$-graded vector space $S=\oplus_{a\in\Z_2^n}S_a$ modulo the Clifford relations $$f_a f_b-(-1)^{\la{a},{b}\ran}f_bf_a=h_{ab}\cdot 1\,,$$ where $h:S\otimes S\to\mathbb{F}$ is
a (non-degenerate) $\Z_2^n$-antisymmetric bilinear form on $\zP S$ which gives rise to a $\Z_2^n$-graded Poisson bracket on $C^\infty(\zP S^*)$. 
There is an analog of the Moyal formula for the corresponding associative product. Also a generalization of super Lie theory to $\Z_2^n$-super (color) Lie algebras and Lie groups is rather straightforward (cf. \cite{Sche79,KS12}).\medskip

Of course, a theory of superfields defined on $\Z_2^n$-supermanifolds requires a theory of $\Z_2^n$-superintegration, which is still in process of construction, but $\sigma$-models with values in our `superization' $\zP S$, for a $\Z_2^n$-graded vector bundle $S$ concentrated in nonzero degrees, can be directly constructed \cite{Schw93}.
In particular, a `higher' Minkowski space can be defined as the color Lie group $M$ associated with the color Lie algebra $\mathcal{L}=V\ti\zP S^*$, where $V$ is the central part corresponding to the vector part of the Minkowski space $\check M$ and the color Lie bracket is determined by a `color-antisymmetric' morphism $\zG:\zP S\otimes \zP S\to V$. In coordinates,
$$[f_a,f_b]=-2\zG^\zm_{ab}e_\zm,$$
where $\zG^\zm_{ab}=-(-1)^{\la a,b\ran}\zG^\zm_{ba}$.
The form $\tilde \zG:\zP S^*\otimes \zP S^*\to V$ defined from the Clifford relations determines a $\Z_2^n$-symmetric bilinear form on spinor fields $\psi: \check M\to \zP S$, which takes the Dirac form
$$\psi\Dc\psi=\tilde\zG^\zm(\psi,\partial_\zm\psi)=\tilde\zG^{\zm ab}\psi_a\partial_\zm\psi_b\,.$$
It is not clear whether super-Lagrangians using $\psi\Dc\psi$ can have a sound physical meaning, but this is a general problem even for the whole string theory.\medskip

Particularly interesting could be extension of $\op{BV}$ and $\op{AKSZ}$-formalisms \cite{AKSZ97, Schw96} and general concepts like \emph{$Q$-manifolds} or \emph{Courant brackets} to $\Z^n$-gradings and $\Z_2^n$-parities. Lie groupoid or Lie algebroid structures, compatible with multi-gradations and induced multi-parities in the spirit of \cite{BGG14,BGG16},
can also provide a fruitful frame for this new supergeometry.\medskip

Note that canonical symplectic structures on iterated bundles like $\T^*\T M$ are multi-homogeneous (see \cite{GR09}), so one can develop the concept of Courant brackets as homological multi-homogeneous Hamiltonians on such $\Z_2^n$-supermanifolds (see \cite{Roy02}).\medskip

Observe eventually that the appearance of canonical superizations of $n$-fold vector bundles as canonical models for $\Z_2^n$-supergeometry, which is a variant of the Batchelor-Gawedzki theorem in standard supergeometry, corresponds well with the fact that classical mechanics has recently been recognized as inextricably associated with some double vector bundle structures and their morphisms (e.g. $\ze:\sT^*\sT M\to \sT\sT^*M$) \cite{GGU06,GG08,GG11,Tul77}.

All this shows that extending possible gradations is not merely a formal trick, but it offers a new examination of known concepts and opens new geometric perspectives.

\section{Acknowledgements}

The research of T. Covolo (resp., J. Grabowski, N. Poncin) was founded by the Luxembourgian NRF grant 2010-1, 786207 (resp., the Polish National Science Centre grant DEC-2012/06/A/ST1/00256, the University of Luxembourg grant GeoAlgPhys 2011-2014). The authors are grateful to S. Morier-Genoud and V. Ovsienko (resp., D. Leites, P. Schapira), whose work was the starting point of this paper (resp., for his suggestions and valuable support, for explanations on sheaf-theoretic aspects).

\end{document}